\documentclass{article}
\usepackage{geometry}
\usepackage{graphicx}	
\usepackage[english]{babel}

\usepackage{mathtools}
\usepackage{latexsym, amsfonts,amssymb,mathrsfs,amscd,amsmath,amsthm}
\newtheorem{thm}{Theorem}[section]
\newtheorem{prop}[thm]{Proposition}
\newtheorem{cor}[thm]{Corollary}

\def\rom#1{\emph{#1}}
\def\({\rom(}
\def\){\rom)}
\def\:{\colon}
\def\cH{\mathcal{H}}
\def\cR{\mathcal{R}}
\def\cM{\mathcal{M}}

\def\cK{\mathcal{K}}
\def\cP{\mathcal{P}}
\def\e{\varepsilon}
\def\g{\gamma}
\def\s{\sigma}
\def\d{\delta}
\def\ss{\subset}
\def\border{\partial}
\def\gA{\overline{\gamma(A)}}
\def\gAs{\overline{\gamma\(A\)}}
\def\dis{\operatorname{dis}}
\def\sim{\Delta}
\def\dis{\operatorname{dis}}
\def\diam{\operatorname{diam}}

\begin{document}

\title{The Hausdorff Mapping Is Nonexpanding}
\author{Ivan A. Mikhaylov}
\date{}
\maketitle

\begin{abstract}
In the present paper we investigate the properties of the Hausdorff mapping $\cH$, which takes  each compact metric space to the space of its nonempty closed subspaces. It is shown that this mapping is nonexpanding (Lipschitz mapping with constant $1$).  This paper gives several examples of classes of metric spaces, distances between which are preserved by the mapping $\cH$. We also calculate distance between any connected metric space and any simplex with greater diameter  than the former  one. At the end of the paper we discuss some properties of the Hausdorff mapping  which may help to prove that it is isometric.
\end{abstract}

\section{Introduction}

By $\cH(X)$ we denote the space of all nonempty bounded and closed subspaces of a metric space $X$, endowed with the Hausdorff metric. We will call a metric compact space \emph{compactum}. By $\cM$ we denote the space of all compacta (considered up to an isometry) endowed with the Gromov--Hausdorff metric. The calculation of the Gromov--Hausdorff distance between two arbitrary compacta is, generally, a difficult task, since the most convenient known formula requires calculation of distortions for all possible correspondence (see Preliminaries for definitions and this formula). So it is interesting to study isometric embeddings of $\cM$ into itself. Such isometricies may provide more convenient way of calculation the Gromov--Hausdorff distance between two compacta, by replacing them with their images. Then the calculation of the Gromov--Hausdorff distance between new spaces  may be easier than between the original ones. But there is a lot of obstacles (connected with the properties of $\cM$) in the search  for such isometries. For example, there is the hypothesis that there is no (bijective) isometry from Gromov--Hausdorff space to itself, except identity function. Nevertheless, isometric (non-bijective) mappings may exist. The \emph{Hausdorff mapping} $\cH: \cM \to \cM, X \mapsto \cH(X)$ is a pretender to be such a mapping.

Our goal is to prove that the mapping $\cH$ is nonexpanding and to give examples of spaces, the Gromov--Hausdorff  distance between which are preserved by the Hausdorff mapping. In order to do this we will study the Hausdorff distance. It turns out, that if we have two metric spaces $X$ and $Y$ isometrically embedded in the third one $Z$, then for any compact subset $A \ss X $ the closest (in a sense of Hausdorff distance) subset of $Y$ is a closure of a set, which contains points in $Y$ closest (in a sense of metric in $Z$)  to points in $A$ (see Corollary \ref{cor:cor2}). This fact helps to prove that $\cH$ is Lipschitz mapping with constant 1. 

\section{Preliminaries}

Let $A$ be an arbitrary set. Denote by $\#A$ the \emph{cardinality} of the set $A$. 

Assume that $Z$ is an arbitrary metric space. The distance between its points $x$ and $y$ is denoted by $|xy|$. By ${\diam Z}$ we denote the diameter of $Z$. If $X, Y \subset Z$  are nonempty subsets, then we put $|XY| = \inf\bigl\{|xy| : x \in X, y \in Y\bigr\}$ and $|XY|' = \sup\bigl\{|xy| : x \in X, y \in Y\bigr\}$. If $X = \{x\}$ then we simply write $|xY| = |Yx|$ and $|xY|'=|Yx|'$ instead of  $\bigl|\{x\}Y\bigr| = \bigl|Y\{x\}\bigr|$ and  $\bigl|\{x\}Y\bigr|' = \bigl|Y\{x\}\bigr|'$, respectively. 

For nonempty  $X, Y \subset Z$ let $$|XY|_{Z} = \max\bigl\{\sup\limits_{x\in X}|xY|, \sup\limits_{y\in Y}|yX|\bigr\}.$$ This value is called the \emph{Hausdorff distance between X and Y}. The family of all nonempty bounded and closed subsets of the metric space $Z$ we denote by $\cH(Z)$ and call the \emph{Hausdorff hyperspace}. It is well-known~\cite{burago} that the space $\cH(Z)$ with the Hausdorff distance forms a metric space. 

Let $X$ and $Y$ be metric spaces. A triple $(X',Y',Z)$ consisting of a metric space $Z$ and its subsets $X'$ and $Y'$ isometric  to $X$ and $Y$, respectively, is called a \emph{realization of the pair} $(X,Y)$. \emph{The Gromov--Hausdorff distance } $d_{GH}(X,Y)$ \emph{between} $X$ \emph{and} $Y$ is the infinum of real numbers $r$, for
which there exists a realization $(X',Y',Z)$ of $(X,Y)$ with $|X'Y'|_Z \le r$. The function $d_{GH}$, being restricted to the set $\cM$ of all isometry classes of compacta (compact metric spaces), forms a metric~\cite{burago}.

\begin{prop}[\cite{burago}]\label{prop:basic}
For any metric spaces $X$ and $Y$ it holds
\begin{enumerate}
\item $d_{GH}(X,Y) \le \frac{1}{2}\max\{\diam X, \diam Y\}$\rom{;}
\item if  $X$ is a one-point space then $d_{GH}(X,Y) = \frac{1}{2}\diam Y$.
\end{enumerate}
\end{prop}

Let $X$ and $Y$ be arbitrary nonempty sets. Recall that a \emph{relation} between the sets $X$ and $Y$ is a subset of the Cartesian product $X  \times Y$. The set of all nonempty relations between $X$ and $Y$ we denote by $\cP(X, Y )$. Let us look at each relation $\s \in \cP(X, Y )$ as a multivalued mapping whose domain may be less than the whole X. Then, similarly with the case of mappings, for any $x \in X$ and any $A \ss X$ there are defined their images $\s(x)$ and $\s(A)$, and for any $y \in Y$ and any $B \ss Y$ their preimages $\s^{-1}(y)$ and $\s^{-1}(B)$, respectively. $A$ relation $R \in \cP(X, Y)$ is called a \emph{correspondence} if the restrictions of the canonical projections $\pi_X : (x, y) \mapsto x$ and $\pi_Y : (x, y) \mapsto y$ onto $R$ are surjective. The set of all correspondences between $X$ and $Y$ we denote by $\cR(X, Y )$.

Let $X$ and $Y$ be arbitrary metric spaces. The \emph{distortion} $\dis \s$ \emph{of a relation} $\s \in \cP(X, Y)$ is the value $$\dis \s = \sup\Big\{\big||xx'|-|yy'|\big|:(x,y),(x',y')\in \s\Big\}.$$

\begin{prop}[\cite{burago}]\label{prop:equality}
For any compacta $X$ and $Y$ it holds
$$d_{GH}(X,Y)=\frac{1}{2}\inf\big\{\dis R: R \in \cR(X,Y)\big\}.$$
\end{prop}

We call a metric space $M$ a \emph{simplex} if all its nonzero distances are the same. Notice that a simplex
$M$ is compact iff it is finite. A simplex consisting of $n$ vertices on the distance $t$ from each other
is denoted by $t\sim_n$. For $t=1$ the space $t\sim_n$ is denote by  $\sim_n$ for short.

\begin{prop}[\cite{distances}]\label{prop:dist1}
For positive integer $p,q$ and real $t,s > 0$ it holds
$$2d_{GH}(t\sim_p,s\sim_q) = \begin{cases} |t-s|, & \mbox{if $p=q$,} \\ \max\{t, s-t\},& \mbox{if $p>q$,} \\ \max\{s, s-t\},& \mbox{if $p<q$.}\end{cases}$$
\end{prop}

\begin{prop}[\cite{distances}]\label{prop:dist2}
Let $M$ be a finite metric space, $n=\#M$. Then for every $m \in \mathbb{N}, m>n$ and $t > 0$ we have
$$2d_{GH}(t\sim_m, M)=\max\{t,\diam M - t\}.$$
\end{prop}

For any metric space $M$ we put $$\e(M)=\inf\bigr\{|xy|: x, y \in M, x\ne y\bigl\}.$$

\begin{prop}[\cite{distances}]\label{prop:dist3}
Let $M$ be a finite metric space, $n=\#M$, then 
$$2d_{GH}(t\sim_n, M) = \max\bigr\{t-\e(M), \diam M -t\bigl\}.$$
\end{prop}

\section{Main Results}
It is well-known~\cite{burago} that if $X$ is a compactum, then $\cH(X)$ is also compact and is equal to the family of all subcompacta of $X$. Consider the \emph{Hausdorff mapping}  $\cH\:\cM  \to \cM$, taking each compactum $X$ to its Hausdorff hyperspace. Let us prove that  the mapping $\cH$ is nonexpanding (and, consequently, continuous).

Let $X, Y\in \cM$ be nonisometric  compacta. Let $X, Y$ be isometrically embedded into metric space $Z$. Due to compactness  of $Y$, for any $x \in X$ there is $y'\in Y$ such that $|xY|=|xy'|$. Let us denote $y'$ by $\g(x)$, thereby  defining the mapping $\g\:X\to Y$. Let $A \subset X$ be an arbitrary closed subset then $\gA \in \cH(Y)$, where $\overline B$ denotes the closure of $B \ss Z$. Now we describe some properties of $\gA$ (notice that $\g$ is defined only for compacta).

\begin{prop}\label{prop:prop1}
For any $A \in \cH\(X\)$ and $a \in A$ we have $\bigr|a\gAs\bigl| ~= ~\bigr|a\g\(a\)\bigl|.$
\end{prop}
\begin{proof}
By definition  of $\g$, for any $y \in Y$ it is true that $|ay|\ge~ \bigr|a\g(a)\bigl|$, therefore for any $a \in A$ it holds $
\bigr|a\gA\big|~= \inf \limits_{y \in \gA}|ay| =~\bigr|a\g(a)\bigl|$ because $\g(a) \in \gA$.
\end{proof}

\begin{prop}\label{prop:prop2}
For every $A \in \cH\(X\)$ it holds $\bigr|A\gAs\bigl|_{Z} ~= \sup \limits_{a\in A}\bigr|a\g\(a\)\bigl|$.
\end{prop}
\begin{proof}
Let $y \in \gA$. Consider two possibilities.
\begin{enumerate}
\item  If $y \in \g(A)$ then there is $a \in A$ such that $ y=\g(a)$. But $|yA| = \inf \limits_{a \in A}|ya|$, so $|yA| \leq ~\bigr|a\g(a)\bigl|$.
 
\item If $y \in \border\g(A)$ then there is a sequence $\{a_{n}\}_{n=1}^{\infty}$ of points from  $A$ such that $\g(a_{n}) \to y$ as $n \to \infty$. Since in compact metric space any sequence of points has convergent subsequence we can, without loss of generality, assume that $a_n \to a \in A$. Suppose that  $\bigr|a\g(a)\bigl|~ < |ay|$ then $\d := |ay| - \bigr|a\g(a)\bigl|~ >0$. There is a number $k >0$ such that $|aa_{k}| < \frac{\d}{3}$ and $\bigr|y\g(a_{k})\bigl|~ < \frac{\d}{3}$. It follows from the triangle  inequality for the triangle $a\g(a)a_{k}$ that 
\begin{multline*}
$$\bigr|a_{k}\g(a)\bigl|~ \le ~\bigr|a\g(a)\bigl| + |aa_{k}| <~ \bigr|a\g(a)\bigl| + \frac{\d}{3} = |ay| - \d + \frac{\d}{3} = |ay| - \frac{2}{3}\d<\\
< |ay| - |aa_{k}| - \bigr|y\g(a_{k})\bigl|~ \le ~\bigr|a_{k}\g(a_{k})\bigl|,$$
\end{multline*} 
that contradicts  the definition of $\g(a_{k})$. Hence, $|ay| = ~\bigr|a\g(a)\bigl|$ and, consequently, $|yA| \le~ \bigr|a\g(a)\bigl|$. 
\end{enumerate}
So for every $ y \in \gA$ there is $a \in A$ such that $|yA| \leq~ \bigr|a\g(a)\bigl|$, thus $ \sup \limits_{y \in \gA}|yA| \le  \sup \limits_{a \in A}\bigr|a\g(a)\bigl|$. Hence, by Proposition~\ref{prop:prop1} we have 
$$\bigr|A\gA\bigl|_{Z}~ = \max\Bigr\{\sup \limits_{y \in \gA}|yA|,~ \sup \limits_{a \in A}\bigr|a\g(a)\bigl|\Bigl\}~ =\sup \limits_{a \in A}\bigr|a\g(a)\bigl|,$$
as required.
\end{proof}

\begin{cor}\label{cor:cor1}
For every $A\in \cH(X)$ we have $\bigr|A\gA\bigl|_Z  ~\le \sup\limits_{x\in X}\bigr|x\g(x)\bigl|$.
\end{cor}

\begin{prop}\label{prop:prop3}
For any $A\in \cH(X)$ and $B\in\cH(Y)$ it holds $|AB|_{Z}\ge ~\bigr|A\gA\bigl|_{Z}$.
\end{prop}
\begin{proof}
For every $a\!\in\! A$ we have $|aB| = \inf \limits_{y \in B}|ay| \geq \inf \limits_{y \in Y}|ay|= ~\bigr|a\g(a)\bigl|$. Hence, $\sup \limits_{a \in A} |aB| \geq \sup \limits_{a \in A}\bigr|a\g(a)\bigl|= \bigr|A\overline{\g(A)}\bigl|_{Z} $, therefore $$|AB|_{Z} = \max{\bigr\{\sup \limits_{a \in A}|aB|,~ \sup \limits_{b \in B}|bA|\bigr\}} \geq ~\bigr|A\overline{\gamma(A)}\bigl|_{Z},$$
as required.
\end{proof}

\begin{cor}\label{cor:cor2}
For every $A\in \cH(X)$ we have $\bigr|A\cH(Y)\bigl| = \bigr|A\gAs\bigl|_{Z}$.
\end{cor}

\begin{thm}\label{thm:main}
Let $X$ and $Y$ be compacta isometrically embedded into a metric space $Z$. Then $$\bigr|\cH(X)\cH(Y)\bigl|_{\cH(Z)}  = |XY|_{Z}.$$
\end{thm}
\begin{proof}
Corollary \ref{cor:cor1} implies that $\sup \limits_{A \in \cH(X)}\bigr|A\gA\bigl|_{Z} \le \sup \limits_{x\in X} \bigr|x\g(x)\bigl|=\sup \limits_{x\in X} |xY|$. By Corollary \ref{cor:cor2}, we have $\sup\limits_{A\in\cH(X)}\bigr|A\cH(Y)\bigl|\le\sup\limits_{x\in X}|xY|$.
\\
On the other hand, for every $x \in X$ it holds $\bigr|x\g(x)\bigl| = \bigr|\{x\}\{\g(x)\}\bigl|_{Z} = \bigr|\{x\}\overline{\g(\{x\})}\bigl|_{Z}$. Hence, $$\sup \limits_{x\in X} |xY| =\sup \limits_{x\in X} \bigr|x\g(x)\bigl| \le \sup \limits_{A \in \cH(X)}\bigr|A\gA\bigl|_{Z} = \sup \limits_{A \in \cH(X)} \bigr|A\cH(Y)\bigl|.$$
\\
Thus $\sup \limits_{A \in \cH(X)}\bigr |A\cH(Y)\bigl| = \sup \limits_{x\in X} |xY|$.
\\
Similarly, we get $\sup \limits_{B \in \cH(Y)} \bigr|B\cH(X)\bigl| = \sup \limits_{y\in Y} |yX|$.
\\
This means that
\begin{multline*}
$$\bigr|\cH(X)\cH(Y)\bigl|_{\cH(Z)} =\max\Bigr\{\sup \limits_{A \in \cH(X)} \bigr|A\cH(Y)\bigl|, \sup \limits_{B \in \cH(Y)} \bigr|B\cH(X)\bigl|\Bigl\}=\\ 
= \max\bigr\{\sup \limits_{x\in X} |xY|, \sup \limits_{y\in Y} |yX|\bigl\}= |XY|_{Z}.$$\qedhere
\end{multline*} 
\end{proof}

\begin{thm}\label{thm:result}
For any compacta $X$ and $Y$ we have $$d_{GH}\bigl(\cH(X), \cH(Y)\bigr)\le d_{GH}(X, Y).$$
\end{thm}
\begin{proof}
From Theorem \ref{thm:main} it follows that if we have a realization $(X, Y, Z)$ of the pair $(X, Y)$ such that $|XY|_{Z} < g$, then $\bigl(\cH(X), \cH(Y), \cH(Z)\bigr)$ is a realization of the pair $\bigl(\mathcal{H}(X), \mathcal{H}(Y)\bigr)$ such that $\bigr|\mathcal{H}(X)\mathcal{H}(Y)\bigl|_{\mathcal{H}(Z)} < g$. Hence, by the definition of the Gromov--Hausdorff distance we have $d_{GH}\bigl(\cH(X), \cH(Y)\bigr)\le d_{GH}(X, Y).$
\end{proof}

\begin{cor}\label{cor:continuous}
The mapping $\cH$ is continuous.
\end{cor}

\begin{prop}\label{prop:diam}
For any compactum $X$ it holds \begin{enumerate}
\item $\diam X=\diam\cH(X);$
\item $\e(x)=\e\bigr(\cH(X)\bigl).$
\end{enumerate}
\end{prop}
\begin{proof}
For any $A,B\in \cH(X)$, due to compactness there are points $a \in A$ and $b \in B$ such that $|AB|_{\cH(X)} = |ab| = \bigr|\{a\}\{b\}\bigr|$. So, $\diam\cH(X) = \sup\limits_{A,B\in \cH(X)}{|AB|_{\cH(X)}} = \sup\limits_{a,b \in X}|ab|=\diam X$. Replacing $\sup$ with $\inf$ we get the second statement.
\end{proof}

\begin{cor}\label{cor:lip}
The mapping  $\cH$ is nonexpanding.
\end{cor}
\begin{proof}
By Theorem \ref{thm:result}, the mapping $\cH$ is a Lipschitz mapping with constant $C \le1$. Let $X=\{x\}$ be a one-point space and $Y$ be an arbitrary compact metric space. Notice, that 1) $\cH\bigl(\{x\}\bigr)=\{x\}$; 2) by Proposition \ref{prop:basic} the Gromov--Hausdorff distance between one-point space any other metric space equals half of the diameter of the latter one; 3) by Proposition \ref{prop:diam}, $\diam\cH(Y) = \diam Y$. So $d_{GH}\bigl(\{x\},Y\bigr) = d_{GH}\Bigl(\cH\bigl(\{x\}\bigr),\cH(Y)\Bigr)$. Hence, $C\ge1$, and we finally get $C=1$.
\end{proof}

\section{Examples}
In this section we construct several families of spaces such that the Gromov--Hausdorff distances between their elements are preserved by the mapping $\cH$.  

\begin{prop}\label{prop:example1}
The mapping $\cH$ does not change the distance between simplices.
\end{prop}
\begin{proof}
Notice that $\cH(t\sim_p) = t\sim_{2^p-1}$. Therefore, according to the formula from Proposition \ref{prop:dist1}, $\cH$ does not change  the Gromov--Hausdorff distance between  finite simplices.
\end{proof}

\begin{prop}\label{prop:example2}
Let $M$ be a finite metric space, $m=\#M, n\ge m$, and $t > 0$. Then
$$d_{GH}(t\sim_n, M) = d_{GH}\bigr(\cH(t\sim_n), \cH(M)\bigl).$$        
\end{prop}
\begin{proof}
By Proposition \ref{prop:diam} $\e(M)=\e\bigl( \cH(M)\bigr)$ and  $\diam X= \diam \cH(X)$. Then the result follows from Proposition \ref{prop:dist2} and Proposition \ref{prop:dist3}.
\end{proof}

The family of all subcompacta of $X$ we denote by $\cK(X)$. In order to construct the next family of examples we will use the fact that mapping $\cH$ preserves connectedness. Though this fact follows from work~\cite{borsuk}, in which Borsuk and Mazurkiewicz proved that $\cK(X)$ is arcwise connected when $X$ is continuum (i.e., connected compact metric space), we will prove it nevertheless, especially since  our proof won't use compactness of $X$. In our proof  we will use the following estimation on the Hausdorff distance.

\begin{prop}\label{prop:noneq}
Let $X=\{x_1,\ldots,x_n\}$ and $Y=\{y_1,\ldots,y_k\}$ be finite subsets of a metric space $Z$. Then for every permutation $\s$ we have $|XY|_Z\le\max\limits_{\substack{1\le i \le n \\ 1\le j \le k}}|x_iy_{\s(j)}|$.
\end{prop}
\begin{proof}
For any point $x_i$ we have $|x_iY|\le|x_iy_{\s(j)}|$ , so $\sup\limits_{1\le i \le n}|x_iY|  \le \max\limits_{\substack{1\le i \le n \\ 1\le j \le k}}|x_iy_{\s(j)}|$. Similarly $\sup\limits_{1\le j \le k}|y_{\s(j)}X| \le \max\limits_{\substack{1\le i \le n \\ 1\le j \le k}}|x_iy_{\s(j)}|$. Hence, $|XY|_Z\le\max\limits_{\substack{1\le i \le n \\ 1\le j \le k}}|x_iy_{\s(j)}|$ as required.
\end{proof}

\begin{prop}\label{prop:connect}
If $X$ is a connected metric space then $\cK\(X\)$ is also connected.
\end{prop}
\begin{proof}
If $X$ is a connected metric space then for any $n\ge1$ the  metric space $ (X^n, d_n)$, where $d_n(x,x')  = \max\limits_{1\le i \le n}{d(x_i, x'_i)}$, is also connected. Let $\pi_n \: X^n \to \cK(X)$ defined as $\pi_n(x) = \bigcup\limits_{i=1}\limits^{n}{\{x_i\}}$. It is obvious that  $\pi_n(X^n)$ is the set of subsets of $X$  with cardinality at most $n$. By Proposition \ref{prop:noneq}, $\pi_n$ is $1$-Lipschitz, thus it is continuous. Therefore, $\pi_n(X^n)$ is connected, as an image of continuous mapping. Hence, $\bigcup\limits_{n}{\pi_n(X^n)}$ is connected because $\pi_n(X^n)\cap\pi_k(X^k)=\pi_k(X^k) \ne\emptyset$ for any $k \le n$. But finite sets are dense in $\cK(X)$, thus $\cK(X) = \overline{\bigcup\limits_{n}{\pi_n(X^n)}}$ is also connected as closure of  connected space.
\end{proof}

Now we prove the following property of connected metric spaces.

\begin{prop}\label{prop:property}
If $X = \bigsqcup\limits_{i = 1}\limits^{n} X_i$ is a connected metric space then for any $i \le n$ there is $j\ne i$ such that $|X_iX_j| = 0$.
\end{prop}
\begin{proof}
Suppose the opposite, i.e., there exists a number $i$ such that $|X_iX_j|>0$ for any number $j\ne i, 1 \le j\le n$. Then $\bigr|X(X\setminus X_i)\bigl|=\d > 0$. So, the open neighborhoods of $X_i$ and $X\setminus X_i$ with radius $\d/2$ don't intersect each other, which means that  $X$ is disconnected .\qedhere
\end{proof}

\begin{prop}\label{prop:con_dist}
Let $X$ be a connected metric space. Then for every real $t$ greater than $\diam X$, and any positive integer $p$ we have
$d_{GH}(t\sim_p, X) = \frac{1}{2}t.$
\end{prop}
\begin{proof}
Let $1,2,\dots ,p$ be points of $t\sim_p$, and $\d$ be an arbitrary positive real number. Then there is a correspondence $R\in \cR(t\sim_p, X)$ such that  $\dis R \le 2d_{GH}(t\sim_p, X) + \d$. We put $X_i =R(i)$. Assume that there are some integers $i, j$ such that $ X_i\cap X_j \ne \emptyset$. Since $\diam X \le t$ for any  $x \in R(i)$, and $ x' \in R(j)$, we have $0 \le t - |xx'| \le t$, so $\dis R \ge t$. If for any integers $i,i'$ it holds $X_i\cap X_{i'} = \emptyset$ then by Proposition \ref{prop:property}, there exists an integer $j$ such that $|X_iX_j| = 0$. This means that  for any real $\d ' > 0$ we can find points $x_i \in X_i$ and $x_j \in X_j$ such that $|x_ix_f|\le \d'$. Similarly as before, $\dis R \ge t - \d'$, but $\d'$ is an arbitrary real number, so again $\dis R \ge t$. As a result we have $2d_{GH}(t\sim_p, X)\ge t - \d$, but once again the choice of $\d$ is arbitrary, hence $d_{GH}(t\sim_p, X)\ge \frac{1}{2}t.$ By Proposition \ref{prop:basic}, we have $d_{GH}(t\sim_p, X)\le \frac{1}{2}t$, so $d_{GH}(t\sim_p, X)= \frac{1}{2}t$.
\end{proof}

\begin{cor}\label{cor:example3} 
Let $X$ be a continuum. Then for every real $t$ such that $t \ge \diam X$, and any positive integer $p$, it holds
$$d_{GH}(t\sim_p, X) = d_{GH}\bigl(\cH(t\sim_p), \cH(X)\bigr).$$
\end{cor}

\section{Additional Properties of the Hausdorff Mapping}
In this section we present some properties which may help to prove or to find a counterexample to hypothesis that $\cH$ is isometric.

Notice the following obvious fact:

\begin{prop}\label{prop:isometric}
Suppose that the Hausdorff mapping preserves distance between points of some dense subspace  of $\cM$. Then $\cH$ is isometric.
\end{prop}

\begin{cor}\label{cor:finite}
If restriction of $\cH$ on the set of all finite metric space is isometric then it is isometric on the whole $\cM$.  
\end{cor}

Finite metric space is called \emph{space in general position}, if its all nonzero distances are distinct and all triangle inequalities are strict.

\begin{cor}\label{corollary:general}
If the Hausdorff mapping preserves distances between spaces in general position then it is isometric.
\end{cor}

\end{document}